\documentclass[reqno]{amsart}
\usepackage{amsmath, amsthm, amssymb, mathabx}
\begin{document}
\newtheorem{theo}{Theorem}
\newtheorem{defin}[theo]{Definition}
\newtheorem{rem}[theo]{Remark}
\newtheorem{lemme}[theo]{Lemma}
\newtheorem{cor}[theo]{Corollary}
\newtheorem{prop}[theo]{Proposition}
\newtheorem{exa}[theo]{Example}
\newtheorem{exas}[theo]{Examples}
%
%
\subjclass[2010]{Primary 37J45  Secondary 35B38 70H05}
\keywords{ Hamiltonian systems, Homoclinic orbits, Ground state, Generalized fountain theorems, Generalized Nehari manifold.}
\thanks{}
\title[Homoclinic orbits of first-order superquadratic ...]{Homoclinic orbits of first-order superquadratic  Hamiltonian systems}

\author[C. J. Batkam]{Cyril J. Batkam}
\address{Cyril J. Batkam \newline
D\'epartement de math\'ematiques,
\newline
Universit\'e de Sherbrooke,
\newline
Sherbrooke, (Qu\'ebec),
\newline
J1K 2R1, CANADA.
\newline {\em Fax:} (819)821-7189.}
\email{cyril.joel.batkam@Usherbrooke.ca}
\maketitle

\begin{abstract}
In this article, we study the existence of homoclinic orbits for the first-order Hamiltonian system
\begin{equation*}
    J\dot{u}(t)+\nabla H(t,u(t))=0,\quad t\in\mathbb{R}.
\end{equation*}
Under the Ambrosetti-Rabinowitz's superquadraticy condition, or no Ambrosetti-Rabinowitz's superquadracity condition, we present two results on the existence of infinitely many large energy homoclinic orbits when $H$ is even in $u$. We apply the generalized (variant) fountain theorems due to the author and Colin. Under no Ambrosetti-Rabinowitz's superquadracity condition, we also obtain the existence of a ground state homoclinic orbit by using the method of the generalized Nehari manifold for strongly indefinite functionals developed by Szulkin and Weth.
\end{abstract}
%


\section{Introduction}
In this paper, we are interrested in the existence and also the multiplicity of homoclinic orbits of the first order Hamiltonian system
\begin{equation}\label{hamilto}\tag{HS}
     J\dot{u}(t)+\nabla H(t,u(t))=0,\quad t\in\mathbb{R},
\end{equation}
where $u\in\mathbb{R}^{2N}$, $J=\left(
           \begin{array}{cc}
             0 & -I_N \\
             I_N & 0 \\
           \end{array}
         \right)
$ is the standard symplectic structure on $\mathbb{R}^{2N}$, $H:\mathbb{R}\times \mathbb{R}^{2N}\rightarrow \mathbb{R}$ is $1-$periodic with respect to the $t-$variable and $\nabla H$ is the gradient of $H$ with respect to $u$. We consider the case which $H$ has the form
$$H(t,u)=\frac{1}{2}A(t)u\cdot u+W(t,u),$$
where the \textbf{dot} denotes the inner product of $\mathbb{R}^{2N}$, $A:\mathbb{R}\rightarrow\mathbb{R}^{4N^2}$ is a $2N\times2N$ symmetric matrix-valued function and $W:\mathbb{R}\times\mathbb{R}^{2N}\rightarrow\mathbb{R}$ satisfy the following conditions.
\begin{enumerate}
  \item [$(A_0)$] $A\in\mathcal{C}^1(\mathbb{R},\mathbb{R}^{4N^2})$ is $1-$periodic with respect to $t$ and $0$ lies in a gap of the spectrum $\sigma(L)$ of $L:=-J\frac{d}{dt}-A(t)$.\\
  \item [$(W_1)$] $W\in\mathcal{C}^1(\mathbb{R}\times\mathbb{R}^{2N},\mathbb{R})$ is $1-$periodic with respect to $t$ and $W(t,0)=0$ for every $t\in \mathbb{R}$.\\
  \item [$(W_2)$] There exist $c>0$, $2< p< \infty$ such that $|\nabla W(t,u)|\leq c\big(1+b(t)|u|^{p-1}\big)$, where $b>0$, $b\in L^{\infty}(\mathbb{R})\cap L^r(\mathbb{R})$, $\frac{1}{r}+\frac{p}{s}=1$ with $2<s<\infty$.\\
  \item [$(W_3)$] $|\nabla W(t,u)|=\circ(|u|)$ as $|u|\rightarrow0$ uniformly in $t$.
\end{enumerate}
 By homoclinic orbit of \eqref{hamilto} we mean a solution $u$ satisfying
 $$ u\neq0\quad \textnormal{and}\quad u(t)\rightarrow0 \textnormal{ as } t\rightarrow\infty.$$
\par In recent years, the existence and multiplicity of homoclinic orbits for the first order system \eqref{hamilto} were studied extensively by means of critical point theory, see for instance \cite{Chen-Ma,Ding-Jean,Ding-Wil,Ding-Gi,Szul-Zou,Ding-Lee,Ding,Rabi-Tana,Coti.al,Sere,Tan,Wang.al}. In their seminal paper \cite{Coti.al}, Coti Zelati, Ekeland and S\'{e}r\'{e}  first considered \eqref{hamilto} with $A$ a constant matrix, $0$ is not a spectrum point of the Hamiltonian operator $L=-J\frac{d}{dt}-A(t)$, and $W(t,u)$ is convex in $u$ and satisfies the Ambrosetti-Rabinowitz's condition
\begin{equation}\label{AR}\tag{AR}
    \exists\mu>2,\quad 0<\mu W(t,u)\leq u\cdot\nabla W(t,u),\quad\forall u\neq0,
\end{equation}
which is extensively used in the study of superquadratic Hamiltonian systems. They proved the existence of two geometrically distinct homoclinic orbits for \eqref{hamilto}. Subsequently, S\'{e}r\'{e} \cite{Sere} obtained the existence of infinitely homoclinic orbits for \eqref{hamilto} under more general assumption on $W$. In \cite{Tan}, Tanaka removed the convexity assumption and obtained the existence of at least one homoclinic orbit by using a subharmonic approach.
\par In this paper we first show, under condition \eqref{AR}, that \eqref{hamilto} has infinitely many homoclinic orbits. We apply the generalized fountain theorem for strongly indefinite functionals established by the author and Colin \cite{Bat-Col1}. More precisely, we have the following result.
\begin{theo}\label{theo1}
Assume that $(A_0)$, $(W_1)-(W_4)$ are satisfied. If in addition
\begin{enumerate}
 \item [$(W_4)$] $W(t,-u)=W(t,u)$,
  \item[$(W_5)$] $ \quad \exists \mu>\max\{2,p-1\},\, \exists\delta>0 \,\, ;\,\, \delta|u|^\mu\leq\mu W(t,u)\leq u\cdot \nabla W(t,u),$
\end{enumerate}
then \eqref{hamilto} has infinitely many large energy homoclinic solutions.
\end{theo}
\begin{rem}
The existence of infinitely many homoclinic orbits of \eqref{hamilto} under $(A_0)$ and $(W_1)-(W_4)$ was first proved by Ding and Girardi \cite{Ding-Gi} \big(see also \cite{Bar-Szu}\big) by using a generalized linking theorem. They allowed $0$ to be an end point of the spectrum $\sigma(L)$. However, they assumed in addition that
\begin{equation*}
    \exists c_0,\varepsilon_0>0\,\, ;\,\, |\nabla W(t,u+v)-\nabla W(t,u)|\leq c_0|v|(1+|u|^{p-1}),\,\textnormal{ whenever }|v|\leq\varepsilon_0.
\end{equation*}
Moreover, we do not know if the homoclinic orbits they obtained are large energy
solutions of \eqref{hamilto}.
\end{rem}
\par It is well known that condition \eqref{AR} is mainly used, in superquadratic problems, to assure the boundedness of the Palais-Smale sequences of the energy functional, and without it the problem becomes more complicated.
 By applying a generalized linking theorem in the spirit of Krysewski and Szulkin \cite{K-S},  Wang et al. \cite{Wang.al}  obtained, without condition \eqref{AR}, the existence of at least one homoclinic orbit of \eqref{hamilto}. Subsequently, by replacing condition \eqref{AR} with a general superquadratic condition, Chen and Ma \cite{Chen-Ma} obtained the existence of at least one homoclinic orbit of \eqref{hamilto} which is a ground state solution, that is a non zero solution which least energy. They adapted an earlier argument used by Yang \cite{Yang} in the study of ground state  solutions for a semilinear Schr\"{o}dinger equation with periodic potential. By replacing in this paper \eqref{AR} with a general superquadratic condition, we also prove the following results.
 \begin{theo}\label{theo2}
Assume that $(A_0)$, $(W_1)-(W_4)$ are satisfied. Assume in addition that $W$ satisfies the following conditions.
\begin{itemize}
  \item [$(W_6)$] $\big(v\cdot\nabla W(t,u)\big)(u\cdot v)\geq0$.\\
  \item [$(W_7)$] $\exists \gamma>2$ such that $\frac{W(t,u)}{|u|^\gamma}\rightarrow\infty$ as $|u|\rightarrow\infty$, uniformly in $t$.\\
  \item [$(W_8)$] $W(t,u)>0$ and $u\cdot\nabla W(t,u)>2W(t,u)$, $\forall u\neq0$.\\
  \item [$(W_9)$] if $|u|=|v|$, then $W(t,u)=W(t,v)$ and $v\cdot\nabla W(t,u)\leq u\cdot\nabla W(t,u)$, with strict inequality if $u\neq v$.\\
  \item [$(W_{10})$] $|u|\neq |v|$ and $u\cdot v\neq0$ $\Rightarrow$ $v\cdot \nabla W(t,u)\neq u\cdot\nabla W(t,v)$.\\
\end{itemize}
Then \eqref{hamilto} has infinitely many large energy homoclinic solutions.
\end{theo}
\begin{theo}\label{theo3}
If $(A_0)$, $(W_1)-(W_3)$, $(W_6)-(W_{10})$ are satisfied, then \eqref{hamilto} has a homoclinic orbit which is a ground state solution.
\end{theo}
\begin{rem}
In Theorem \ref{theo3}, the assumptions $(W_2)$ and $(W_7)$ can be weaken by taking $b\equiv1$ and $\gamma=2$ respectively.
\end{rem}
As far as we know Theorem \ref{theo2} is new. It will be proved by using the generalized variant fountain theorem due to author and Colin \cite{Bat-Col2}, which combines the $\tau-$topology of Kryszewski and Szulkin \cite{K-S} with the idea of the monotonicity trick for strongly indefinite functionals inspired by Jeanjean \cite{LJJ}. Theorem \ref{theo3} was first proved by Chen and Ma \cite{Chen-Ma}. They applied a generalized weak linking theorem due to Schechter and Zou \cite{S-Zou}. In this paper we follow a different approach, which is based on the method of the generalized Nehari manifold for strongly indefinite functionals inspired by Pankov \cite{Pan}, and developed recently by Szulkin and Weth \cite{Sz-W,S-W}. This approach is much more direct and simpler.
\par The paper is organized as follows. The variational framework for the study of \eqref{hamilto} will be stated in section \ref{section1}, while the existence of infinitely many large energy homoclinic orbits will be proved in Section \ref{section2}. Finally, in Section \ref{section3} we apply the method of the generalized Nehari manifold to find a ground state homoclinic orbit of \eqref{hamilto}.

%

\section{Variational setting}\label{section1}
Let $X:=H^{\frac{1}{2}}(\mathbb{R},\mathbb{R}^{2N})$ be the fractional Sobolev space of functions $u\in L^2(\mathbb{R},\mathbb{R}^{2N})$ such that
\begin{equation*}
    \int_\mathbb{R}(1+\xi^2)|\mathcal{F}u(\xi)|^2d\xi<\infty,
\end{equation*}
where $\mathcal{F}$ is the Fourier transform. $X$ is a separable Hilbert space with the inner product
\begin{equation*}
    \big<u,v\big>_{\frac{1}{2}}:=\int_\mathbb{R}(1+\xi^2)^{\frac{1}{2}}\mathcal{F}u(\xi)\overline{\mathcal{F}v(\xi)}d\xi, \quad u,v\in X.
\end{equation*}
For $q\in[2,+\infty[$, the Sobolev embedding $X\hookrightarrow L^q(\mathbb{R},\mathbb{R}^{2N})$ is continuous and the embedding $X\hookrightarrow L_{loc}^q(\mathbb{R},\mathbb{R}^{2N})$ is compact \big(see for example \cite{Adams} or \cite{Stuart}\big).\\
Consider the operator $B:X\rightarrow X$ defined by
\begin{equation*}
   \big<Bu,v\big>_{\frac{1}{2}}:=\int_\mathbb{R} \big(-J\dot{u}-A(t)u\big)\cdot vdt.
\end{equation*}
By assumption $(A_0)$, $L=-J\frac{d}{dt}-A(t): L^2(\mathbb{R},\mathbb{R}^{2N})\rightarrow L^2(\mathbb{R},\mathbb{R}^{2N})$ is a selfadjoint bounded operator with domain $D(L)=H^1(\mathbb{R},\mathbb{R}^{2N})$, and the spectrum is unbounded below and above in $H^1(\mathbb{R},\mathbb{R}^{2N})$ \big(see \cite{Stuart}\big). Hence, the space $X$ has the orthogonal decomposition $X=X^+\oplus X^-$, where $X^{\pm}$ are infinite dimensional $B-$invariant subspaces such that the quadratic form $u\in X\mapsto\big<Bu,u\big>$ is negative on $X^-$ and positive on $X^+$. Therefore we can define a new equivalent inner product on $X$ by setting
\begin{equation*}
    \big<u,v\big>:= \big<Bu^+,v^+\big>_{\frac{1}{2}}- \big<Bu^-,v^-\big>_{\frac{1}{2}},\quad u^{\pm},v^\pm\in X^\pm.
\end{equation*}
If $\|\cdot\|$ denotes the corresponding norm, then we have
\begin{equation*}
\int_\mathbb{R} \big(-J\dot{u}-A(t)u\big)\cdot udt=\|u^+\|^2-\|u^-\|^2,\quad \forall u\in X.
\end{equation*}
\par We define on $X$ the functional
\begin{equation*}
     \Phi(u):=\frac{1}{2}\int_\mathbb{R} \big(-J\dot{u}-A(t)u\big)\cdot udt-\int_\mathbb{R}W(t,u)dt.
\end{equation*}

Then
\begin{equation}\label{phii}
    \Phi(u):=\frac{1}{2}\|u^+\|^2-\frac{1}{2}\|u^-\|^2-\int_\mathbb{R}W(t,u)dt.
\end{equation}
\begin{prop}[\cite{Bar-Szu}, Proposition $3.1$]\label{prop}
If $(W_1)$, $(W_2)$ and $(W_3)$ are satisfied, then $\Phi\in \mathcal{C}^1(X,\mathbb{R})$, with
\begin{equation*}
     \big<\Phi'(u),v\big>=\big<u^+,v\big>- \big<u^-,v\big>-\int_\mathbb{R}v\cdot\nabla W(t,u)dt.
\end{equation*}
Moreover, $u\in X$ is a homoclinic orbit of \eqref{hamilto} if and only if it is a non zero critical point of $\Phi$.
\end{prop}
\par Due to the periodicity of $A$ and $W$, if $u=u(t)$ is a homoclinic orbit of \eqref{hamilto}, so are all $g\star u$, $g\in \mathbb{Z}$, where
\begin{equation*}
    g\star u(t):=u(t-g).
\end{equation*}
Therefore the functional $\Phi$ cannot satisfy the Palais-Smale condition at any critical level $c\neq0$. We recall that a functional $\varphi\in \mathcal{C}^1(X,\mathbb{R})$ is said to satisfy the Palais-Smale condition \big(resp. the Palais-Smale condition at level $c\in\mathbb{R}$\big), if every sequence $(u_n)\subset X$ such that $(\varphi(u_n))_n$ is bounded \big(resp. $\varphi(u_n)\rightarrow c$\big) and $\varphi'(u_n)\rightarrow0$, admits a convergent subsequence. Two homoclinic orbits $u$ and $v$ of \eqref{hamilto} are said to be geometrically distinct if the sets $\{g\star u\,;\,g\in\mathbb{Z}\}$ and $\{g\star v\,;\,g\in\mathbb{Z}\}$ are disjoint.

\section{The existence of infinitely many homoclinic orbits}\label{section2}
\subsection{Generalized (variant) fountain theorems}
Let $Y$ be a closed subspace of a separable Hilbert space $X$ endowed with the inner product $(\cdot)$ and the associated norm $\|\cdot\|$.
We denote by $P:X\rightarrow Y$ and $Q:X\rightarrow Z:=Y^\perp$ the orthogonal projections.
\par{} We fix an orthonormal basis $(a_j)_{j\geq0}$ of $Y$ and we consider on $X=Y\oplus Z$ the $\tau-$topology introduced by Kryszewski and Szulkin in \cite{K-S}; that is the topology associated to the following norm
\begin{equation*}
    \vvvert u\vvvert:=\max\Big(\sum\limits_{j=0}^{\infty}\frac{1}{2^{j+1}}|(Pu,a_j)|,\|Qu\|\Big),\,\, u\in X.
\end{equation*}
$\tau$ has the following interesting property \big(see \cite{K-S} or \cite{W}\big): If $(u_n)\subset X$ is a bounded sequence, then
\begin{equation*}
     u_n
\stackrel{\tau}{\rightarrow}u  \Longleftrightarrow Pu_n \rightharpoonup Pu \,\ and \,\ Qu_n \rightarrow Qu.
\end{equation*}

Let $(e_j)_{j\geq0}$ be an orthonormal basis of $Z$. We adopt the following notations:
$$Y_k:=Y\oplus(\oplus_{j=0}^k{\mathbb R} e_j)\quad\quad \textnormal{ and }\quad\quad  Z_k:=  \overline{\oplus_{j=k}^\infty\mathbb{R} e_j }. $$
$$ \displaystyle B_k:=\{u\in Y_k \, \bigl | \, ||u||\leq \rho_k \bigr.\}, \,\,\,\ N_k:=\{u\in Z_k \, \bigl | \, ||u||=r_k \bigr. \} \,\,\ where\,\,\,\ 0<r_k< \rho_k, \,\ k\geq2.   $$
The following abstract critical point theorems are due to the author and F. Colin.
\begin{theo}[Fountain theorem, Batkam-Colin \cite{Bat-Col1}]\label{ft}
 Let $\Phi\in\mathcal{C}^1(X,{\mathbb R})$ be an even functional which is $\tau$-upper semicontinuous and such that $\Phi'$ is weakly sequentially continuous. If there exist $\rho_k>r_k>0$ such that:
             \begin{itemize}
               \item [$(A_1)$]\,\,\,\  {$a_k \, := \, \displaystyle \sup_{\substack{u \in Y_k \\ \|u\| = \rho_k}} \Phi(u) \le 0 $ \,\,\ \textnormal{and} \,\,\ $\displaystyle \sup_{\substack{u \in Y_k \\ \|u\| \leq \rho_k}} \Phi(u)<\infty $}.
               \item [$(A_2)$] \,\,\,\ {$b_k \, := \, \displaystyle \inf_{\substack{u \in Z_k \\ \|u \| = r_k}} \Phi(u) \rightarrow \infty, \, k \rightarrow \infty$.}
             \end{itemize}
Then
\begin{equation*}
    c_{k} :=  \inf_{\gamma \in \Gamma_{k}} \sup_{u\in B_{k}}\Phi(\gamma(u))\,\geq\,b_k,
\end{equation*}
 and there exists a sequence $(u_k^n)_n\subset X$ such that
\begin{align*}
\Phi'(u_k^n)  \to 0 \quad \textnormal{and}\quad \Phi(u_k^n)  \to c_{k} \,\,\,\ as \,\ n\to\infty,
\end{align*}
where
$\Gamma_k$ is the set of maps $\gamma:B_k\rightarrow X$ such that
\begin{itemize}
  \item [(a)] $\gamma$ is odd and $\tau-$continuous, and $\gamma_{\mid_{\partial B_k}}=id,$
  \item [(b)] every $u\in int(B_k)$ has a $\tau-$neighborhood $N_u$ in $Y_k$ such that $(id-\gamma)(N_u\cap int(B_k))$ is contained in a finite-dimensional subspace of $X$,
   \item [(c)] $\Phi(\gamma(u))\leq\Phi(u)$ $\forall u\in B_k$.
\end{itemize}
\end{theo}
\begin{theo}[Variant fountain theorem, Batkam-Colin \cite{Bat-Col2}]\label{vft}
Let the family of $\mathcal{C}^1$-functionals
\begin{equation*}
    \Phi_\lambda:X\rightarrow\mathbb{R},\quad\Phi_\lambda(u):=L(u)-\lambda J(u), \,\,\,\,\,\,\,\,\ \lambda\in[1,2],
\end{equation*}
such that
 \begin{enumerate}
\item [$(B_1)$] $\Phi_\lambda$ maps bounded sets to bounded sets uniformly for $\lambda\in[1,2]$, and $\Phi_\lambda(-u)=\Phi_\lambda(u)$ for every $(\lambda,u)\in[1,2]\times X$.
\item [$(B_2)$] $J(u)\geq0$ for every $u\in X$; $L(u)\rightarrow\infty$ or $J(u)\rightarrow\infty$ as $\|u\|\rightarrow\infty$.
\item [$(B_3)$] For every $\lambda\in[1,2]$, $\Phi_\lambda$ is $\tau$-upper semicontinuous  and $\Phi'_\lambda$ is weakly sequentially continuous.
\end{enumerate}
If there are $0<r_k<\rho_k$ such that
\begin{equation*}
      b_k(\lambda) \, := \, \displaystyle \inf_{\substack{u \in Z_k \\ \|u \| = r_k}}\Phi_\lambda(u)\,\,\geq\,\,a_k(\lambda) \, := \, \displaystyle \sup_{\substack{u \in Y_k \\ \|u\| = \rho_k}} \Phi_\lambda(u) \,\,\forall\lambda\in[1,2],
\end{equation*}
 then
\begin{equation*}
     c_{k}(\lambda) :=  \inf_{\theta \in \Theta_{k}(\lambda)} \sup_{u\in B_{k}} \Phi_\lambda
\bigl( \theta(u)  \bigr)\,\geq\,b_k(\lambda)\,\,\, \forall\lambda\in[1,2].
\end{equation*}
 Moreover, for a.e $\lambda\in[1,2]$ there exists a sequence $(u_k^n(\lambda))_n\subset X$ such that
\begin{equation*}
    \sup_{\substack{n}}\|u_k^n(\lambda)\|<\infty, \quad \Phi'_\lambda (u_k^n(\lambda))\rightarrow 0 \,\ and \,\ \Phi_\lambda(u_k^n(\lambda))\rightarrow c_k(\lambda)\,\,\textnormal{as}\,\,n\rightarrow\infty.
\end{equation*}
Where
$\Theta_k(\lambda)$ is the class of maps $\theta:B_k\rightarrow X$ such that
\begin{itemize}
  \item [(a)] $\theta$ is odd and $\tau-$continuous, and $\theta_{\mid_{\partial B_k}}=id,$
  \item [(b)] every $u\in int(B_k)$ has a $\tau-$neighborhood $N_u$ in $Y_k$ such that $(id-\theta)(N_u\cap int(B_k))$ is contained in a finite-dimensional subspace of $X$,
   \item [(c)] $\Phi_\lambda(\theta(u))\leq\Phi_\lambda(u)$ $\forall u\in B_k$.
\end{itemize}
\end{theo}
\par In the following two subsections, we set $Y=X^+$ and $Z=X^-$.
\subsection{The case of Ambrosetti-Rabinowitz condition}
In this subsection we, assume that $(A_0)$, $(W_1)-(W_5)$ are satisfied.\\
The functional $\Phi$ reads as follows:
\begin{equation}\label{phi}
    \Phi(u)=\frac{1}{2}\|Qu\|^2-\frac{1}{2}\|Pu\|^2-\int_\mathbb{R}W(t,u)dt.
\end{equation}
We know from Proposition \ref{prop} that $\Phi$ is of class $\mathcal{C}^1$ on $X$ and
\begin{equation}\label{phiprime}
    \big<\Phi'(u),v\big>=\big<Qu,v\big>- \big<Pu,v\big>-\int_\mathbb{R}v\cdot\nabla W(t,u)dt.
\end{equation}
We have the following lemma.
\begin{lemme}
$\Phi$ is $\tau-$upper semicontinuous on $X$, and $\Phi'$ is weakly sequentially continuous.
\end{lemme}
\begin{proof}
Let $u_n\stackrel{\tau}{\rightarrow}u$ in $X$ and $\Phi(u_n)\geq C\in\mathbb{R}$. Then, by the definition of $\tau$ we have $Qu_n\rightarrow Qu$, and then $(Qu_n)$ is bounded. Since $W\geq0$, we deduce from the inequality $C\leq\Phi(u_n)$ that $(Pu_n)$ is also bounded. hence, $u_n\rightharpoonup u$ in $X$, $u_n\rightarrow u$ in $L^p_{loc}(\mathbb{R},\mathbb{R}^{2N})$, and up to a subsequence $u_n(t)\rightarrow u(t)$ a.e $t\in\mathbb{R}$. It follows from Fatou's lemma and the weakly semicontinuity of the norm $\|\cdot\|$ that $C\leq\Phi(u)$. Hence $\Phi$ is $\tau-$upper semicontinuous.
\par Now assume that $u_n\rightharpoonup u$ in $X$. Then $u_n\rightarrow u$ in $L^p_{loc}(\mathbb{R},\mathbb{R}^{2N})$, and since $b(t)\leq\|b\|_\infty$ a.e. $t$, we deduce from Theorem $A.2$ of \cite{W} that $\nabla W(t,u_n)\rightarrow\nabla W(t,u)$ in $L^{p/(p-1)}_{loc}(\mathbb{R},\mathbb{R}^{2N})$. hence $\big<\Phi'(u_n),v\big>\rightarrow\big<\Phi'(u),v\big>$ for all $v\in C^\infty_c(\mathbb{R},\mathbb{R}^{2N})$. We then deduce by density that $\Phi'$ is weakly sequentially continuous.
\end{proof}
\begin{lemme}\label{psborne}
Every Palais-Smale sequence for $\Phi$ is bounded.
\end{lemme}
\begin{proof}
Let $(u_n)\subset X$ and $d\in\mathbb{R}$ such that $\sup |\Phi(u_n)|\leq d$ and $\Phi'(u_n)\rightarrow0$.\\
It follows from $(W_5)$ that
\begin{equation*}
    \Phi(u_n)-\big<\Phi'(u_n),u_n\big>\geq (\frac{\mu}{2}-1)\delta|u|^\mu_\mu.
\end{equation*}
We deduce that for $n$ big enough
\begin{equation}\label{eq1}
    (\frac{\mu}{2}-1)\delta|u|^\mu_\mu\leq d+\|u_n\|.
\end{equation}
On the other hand $(W_2)$ and $(W_3)$ imply that
\begin{equation}\label{eq2}
    \forall \varepsilon>0, \,\, \exists c(\varepsilon)>0;\,\,\, |\nabla W(t,u)|\leq \varepsilon|u|+c(\varepsilon)|u|^{p-1} \,\, a.e.\,t\in\mathbb{R},\,\, \forall u\in\mathbb{R}^{2N}.
\end{equation}
Hence
\begin{eqnarray*}
   \|Qu_n\|^2 &=& \big<\Phi'(u_n),Qu_n\big>+\int_\mathbb{R}Qu_n\cdot W(t,u_n)dt \\
   &\leq&  \|Qu_n\|+\int_\mathbb{R}Qu_n\cdot W(t,u_n)dt\quad(\textnormal{for $n$ big enough})\\
   &\leq&  \|Qu_n\|+\varepsilon\int_\mathbb{R}|Qu_n||u_n|dt+c(\varepsilon) \int_\mathbb{R}|Qu_n||u_n|^{p-1}dt.
\end{eqnarray*}
By the same way we have
\begin{equation*}
    \|Pu_n\|^2\leq\|Pu_n\|+\varepsilon\int_\mathbb{R}|Pu_n||u_n|dt+c(\varepsilon) \int_\mathbb{R}|Pu_n||u_n|^{p-1}dt.
\end{equation*}
And by using the H\"{o}der inequality and the Sobolev embedding theorem we obtain
\begin{equation*}
  \|Qu_n\|^2+\|Pu_n\|^2\leq  \|Qu_n\|+\|Pu_n\|+c_1\varepsilon\|u_n\|^2+c_2c(\varepsilon)\|u_n\||u_n|^{p-1}_\mu.
\end{equation*}
By taking \eqref{eq1} into account we get
\begin{equation*}
  \|Qu_n\|^2+\|Pu_n\|^2\leq  \|Qu_n\|+\|Pu_n\|+c_1\varepsilon\|u_n\|^2+c_3c(\varepsilon)\|u_n\|\big(1+\|u_n\|^\frac{{p-1}}{\mu}\big).
\end{equation*}
Hence,
\begin{equation*}
    (1-c_1\varepsilon)\|u_n\|^2\leq \|Qu_n\|+\|Pu_n\|+c_3c(\varepsilon)\|u_n\|\big(1+\|u_n\|^\frac{{p-1}}{\mu}\big).
\end{equation*}
Since by $(W_5)$ we have $\frac{{p-1}}{\mu}<1$, it then suffices to fix $\varepsilon<\frac{1}{2c_1}$ to conclude.
\end{proof}
The following lemma will be helpful for our arguments. It is a special case
of a more general result due to P. L. Lions \cite{PLL}.
\begin{lemme}\label{lions}
Let $(u_n)$ be a bounded sequence in $X$. If there is $r>0$ such that
\begin{equation*}
    \lim_{\substack{n\rightarrow\infty}}\sup_{\substack{a\in\mathbb{R}}}\int_{-a}^{+a}|u_n|^2=0,
\end{equation*}
then $u_n\rightarrow0$ in $L^q(\mathbb{R},\mathbb{R}^{2N})$ for all $q\in(2,\infty)$.
\end{lemme}
\begin{proof}[\textbf{Proof of Theorem \ref{theo1}}]
\par Let $u\in Y_k$. $(W_5)$ implies that
\begin{equation*}
    \Phi(u)\leq\frac{1}{2}\|Qu\|^2-\frac{1}{2}\|Pu\|^2-c|u|^\mu_\mu.
\end{equation*}
Let $\widehat{Y_k}$ be the closure of $Y_k$ in $L^\mu(\mathbb{R},\mathbb{R}^{2N})$, then there is a continuous projection of $\widehat{Y_k}$ on $\oplus_{j=0}^k{\mathbb R} e_j$, and since all norms are equivalent on the latter space we can find a constant $c_1>0$ such that $c_1\|Qu\|^\mu\leq|u|^\mu_\mu$. It follows that
\begin{equation*}
    \Phi(u)\leq \frac{1}{2}\|Qu\|^2-\frac{1}{2}\|Pu\|^2-c_1\|Qu\|^\mu.
\end{equation*}
This implies that $\Phi(u)\rightarrow-\infty$ as $\|u\|\rightarrow\infty$, and condition $(A_1)$ of Theorem \ref{ft} is therefore satisfied for $\rho_k$ large enough.
\par Now let $u\in Z_k$. We deduce from $(W_2)$ and $(W_3)$ that
\begin{equation}\label{eq3}
    \forall\varepsilon>0,\,\, \exists c_\varepsilon>0\,;\, |\nabla W(t,u)|\leq \varepsilon |u|+c_\varepsilon b(t)|u|^{p-1}.
\end{equation}
It then follows that
\begin{equation*}
    \Phi(u)\geq \frac{1}{2}(1-c\varepsilon)\|u\|^2-c_\varepsilon\int_\mathbb{R}\frac{b(t)}{p}|u|^pdt.
\end{equation*}
By choosing $\varepsilon=\frac{1}{2c}$ we obtain
\begin{equation*}
    \Phi(u)\geq \frac{1}{4}\|u\|^2-c_1\int_\mathbb{R}\frac{b(t)}{p}|u|^pdt.
\end{equation*}
Let
\begin{equation}\label{betaka}
    \beta_k:=\sup_{\substack{u\in Z_k\\\|u\|=1}}\Big(\int_\mathbb{R}\frac{b(t)}{p}|u|^pdt\Big)^{\frac{1}{p}}.
\end{equation}
Then
\begin{equation*}
    \Phi(u)\geq \frac{1}{4}\|u\|^2-c_1\beta^p_k\|u\|^p=\frac{1}{2}\Big(\frac{1}{2}\|u\|^2-c_2\beta^p_k\|u\|^p\Big).
\end{equation*}
If we set $r_k:=\big(c_2p\beta_k\big)^\frac{1}{2-p}$, then for every $u\in Z_k$ such that $\|u\|=r_k$ we have
\begin{equation*}
    \Phi(u)\geq\frac{1}{2}\Big(\frac{1}{2}-\frac{1}{p}\Big)\big(c_2p\beta_k\big)^\frac{2}{2-p}.
\end{equation*}
By Lemma \ref{beta} below, $\beta_k\rightarrow0$ as $k\rightarrow\infty$. Hence assumption $(A_2)$ of Theorem \ref{ft} is satisfied.
\par By applying Theorem \ref{ft}, we obtain the existence of a sequence $(u_k^n)_n\subset X$ such that $\Phi(u_k^n)\rightarrow c_k$ and $\Phi'(u_k^n)\rightarrow0 $ as $n\rightarrow\infty$, for every $k$.\\
We claim that there exist a sequence $(a_n)\subset\mathbb{R}$ and real numbers $r,\gamma>0$ such that for $k$ big enough
\begin{equation}\label{eq4}
    \liminf_{\substack{n\rightarrow\infty}}\int_{-a_n}^{+a_n}|u_k^n|^2\geq\gamma.
\end{equation}
In fact, if the claim is not true then, because $(u_k^n)$ is bounded by Lemma \ref{psborne}, we deduce from Lemma \ref{lions} the existence of a subsequence, still denoted $(u_k^n)$, such that $u_k^n\rightarrow0$ in $L^p(\mathbb{R},\mathbb{R}^{2N})$. By using the H\"{o}lder inequality and \eqref{eq2} we have
\begin{eqnarray*}
  \int_\mathbb{R}Pu_k^n\cdot\nabla W(t,u_k^n)dt &\leq& \varepsilon|u_k^n|_2|Pu_k^n|_2+c(\varepsilon)|u_k^n|_p^{p-1}|Pu_k^n|_p \\
   &\leq& C\big(\varepsilon+c(\varepsilon)|u_n|_p^{p-1}\big),
\end{eqnarray*}
where $C$ is a constant which does not depend on $n$ and $\varepsilon$. It follows that
$$\limsup_{\substack{n\rightarrow\infty}}\int_\mathbb{R}Pu_k^n\cdot\nabla W(t,u_k^n)dt\leq c\varepsilon,$$
 and since $\varepsilon$ is arbitrary we deduce that
 $$\int_\mathbb{R}Pu_k^n\cdot\nabla W(t,u_k^n)dt\rightarrow0 \,\,\textnormal{as}\,\ n\rightarrow\infty.$$
By the same way we show that
$$\int_\mathbb{R}Qu_k^n\cdot\nabla W(t,u_k^n)dt\rightarrow0 \,\ \textnormal{and}\,\ \int_\mathbb{R}W(t,u_k^n)dt\rightarrow0\,\,\textnormal{as}\,\ n\rightarrow\infty.$$
It then follows that
\begin{equation*}
    c_k=\lim_{\substack{n\rightarrow\infty}}\big(\Phi(u_k^n)-\frac{1}{2}\big<\Phi'(u_k^n),u_k^n\big>\big)
   =\lim_{\substack{n\rightarrow\infty}}\int_\mathbb{R}\Big(\frac{1}{2}u_k^n\cdot\nabla W(t,u_k^n)-W(t,u_k^n)\Big)dt=0.
\end{equation*}
We obtain a contradiction by taking $k$ sufficiently large, since $c_k\geq b_k\rightarrow\infty$ as $k\rightarrow\infty$.\\
Now by \eqref{eq4} there exists a subsequence, still denoted $(u_k^n)$, such that for $k$ big enough
\begin{equation*}
    \|u_k^n\|_{L^2\big((\alpha_n-r,\alpha_n+r),\mathbb{R}^{2N}\big)}\geq \frac{\gamma}{2}\quad \forall n.
\end{equation*}
By a standard argument, there is $q_n\in\mathbb{Z}$ such that for $k$ big enough
\begin{equation}\label{eq5}
    \|w_k^n\|_{L^2\big((-r-\frac{1}{2},r+\frac{1}{2}),\mathbb{R}^{2N}\big)}\geq \frac{\gamma}{2}\quad \forall n,
\end{equation}
where $w_k^n:=u_k^n(\cdot-q_n)$. Since both $\Phi$ and $\Phi'$ are invariant under translation, it follows that $\Phi(w_k^n)\rightarrow c_k$ and $\Phi'(w_k^n)\rightarrow0 $ as $n\rightarrow\infty$. By Lemma \ref{psborne} again, the sequence $(w_k^n)$ is bounded. Up to a subsequence, we may suppose that
\begin{equation}\label{eq6}
    w_k^n\rightharpoonup w_k\textnormal{ in }X,\quad w_k^n\rightarrow w_k \textnormal{ in }L^2_{loc}(\mathbb{R},\mathbb{R}^{2N}),\quad w_k^n\rightarrow w_k\textnormal{ a.e. }
\end{equation}
By \eqref{eq5}, $w_k\neq0$ for $k$ large enough, and in view of the weak sequentially semicontinuity of $\Phi'$ we have $\Phi'(w_k)=0$. That is, $w_k$ is a critical point of $\Phi$ and therefore a weak solution of \eqref{hamilto}. Again by \eqref{eq5}, we have for $0<R<\infty$ and $k$ large enough
\begin{equation*}
    \sup_{\substack{a\in\mathbb{R}}}\int_{a-R}^{a+R}|w_k^n-w_k|^2\rightarrow0,\,\,n\rightarrow\infty.
\end{equation*}
Lemma \ref{lions} then implies that $w_k^n\rightarrow w_k$ in $L^p(\mathbb{R},\mathbb{R}^{2N})$. Using this and \eqref{eq2}, one can verify easily that
\begin{equation*}
    \int_\mathbb{R}(w_k^n-w_k)\cdot\nabla W(t,w_k^n-w_k)dt\rightarrow0,\quad \int_\mathbb{R}W(t,w_k^n-w_k)dt\rightarrow0,\quad n\rightarrow\infty.
\end{equation*}
By Brezis-Lieb lemma \cite{B-L}, we also have as $n\rightarrow\infty$
\begin{equation*}
    \int_\mathbb{R}w_k^n\cdot\nabla W(t,w_k^n)dt\rightarrow w_k\cdot\nabla W(t,w_k)dt,\quad \int_\mathbb{R}W(t,w_k^n)dt\rightarrow\int_\mathbb{R}W(t,w_k)dt.
\end{equation*}
By taking the limit $n\rightarrow\infty$ in the expression
\begin{equation*}
    \Phi(w^n_k)=\big<\Phi'(w^n_k),w^n_k\big>+\frac{1}{2}\int_\mathbb{R}w_k^n\cdot\nabla W(t,w^n_k)dt-\int_\mathbb{R}W(t,w_k^n)dt,
\end{equation*}
we therefore deduce that $\Phi(w_k)=c_k$. Since $c_k\geq b_k\rightarrow\infty$, $k\rightarrow\infty$, the theorem is proved.
\end{proof}
\begin{lemme}\label{beta}
Assume that $b>0$, $b\in L^{\infty}(\mathbb{R})\cap L^r(\mathbb{R})$, $\frac{1}{r}+\frac{p}{s}=1$ with $2<p,s<\infty$. Then
\begin{equation*}
    \beta_k=\sup_{\substack{u\in Z_k\\\|u\|=1}}\Big(\int_\mathbb{R}\frac{b(t)}{p}|u|^pdt\Big)^{\frac{1}{p}}\rightarrow0,\,\, \textnormal{as}\,\, k\rightarrow\infty.
\end{equation*}
\end{lemme}
\begin{proof}
Clearly $0\leq\beta_{k+1}\leq\beta_k$, hence $\beta_k\rightarrow\beta\geq0$. For every $k$, there exists $u_k\in Z_k$ such that
\begin{equation*}
    0\leq\beta_k^p-\int_\mathbb{R}\frac{b(t)}{p}|u_k|^pdt<\frac{1}{k}.
\end{equation*}
Up to a subsequence we have $u_k\rightharpoonup u$ in $X$. By the definition of $Z_k$ we have $u=0$. Since $X$ embeds continuously in $L^s(\mathbb{R},\mathbb{R}^{2N})$, the sequence $(u_k)$ is also bounded in $L^s(\mathbb{R},\mathbb{R}^{2N})$. Therefore, there is a constant $C>0$ such that $\||u_k|^p\|_{L^{s/p}(\mathbb{R},\mathbb{R}^{2N})}\leq C$. Let $I_k:=]-k,k[$.
 \begin{equation*}
    b\in L^r(\mathbb{R},\mathbb{R}^{2N})\Rightarrow \|b\|_{L^r(\mathbb{R}\backslash I_k)}\rightarrow0,\quad \textnormal{as }k\rightarrow\infty.
 \end{equation*}
 It follows that for every $\varepsilon>0$, we can find $k_1$ large enough such that $\|b\|_{L^r(\mathbb{R}\backslash I_{k_1})}<\varepsilon$.\\
 Now since the embedding $H^1(I_{k_1},\mathbb{R}^{2N})\hookrightarrow L^p(I_{k_1},\mathbb{R}^{2N})$ is compact, we have $u_k\rightarrow0$ in $L^p(I_{k_1},\mathbb{R}^{2N})$, and since $b(t)\leq\|b\|_\infty$ a.e., we obtain in view of Theorem $A.2$ in \cite{W} that $\int_{I_{k_1}}\frac{b(t)}{p}|u_k|^pdt\rightarrow0$ as $k\rightarrow\infty$. So there is $k_0$ such that
 \begin{equation*}
    \int_{I_k}\frac{b(t)}{p}|u_k|^pdt<\varepsilon,\quad \forall k\geq k_0.
 \end{equation*}
 Since $\frac{1}{p}+\frac{1}{s/p}=1$, we deduce from the H\"{o}lder inequality that
\begin{eqnarray*}
  \int_\mathbb{R}\frac{b(t)}{p}|u_k|^pdt &=& \int_{I_{k_1}}\frac{b(t)}{p}|u_k|^pdt+\int_{\mathbb{R}\backslash I_{k_1}}\frac{b(t)}{p}|u_k|^pdt \\
   &\leq& \int_{I_{k_1}}\frac{b(t)}{p}|u_k|^pdt+\frac{1}{p}\Big(\int_{\mathbb{R}\backslash I_{k_1}}(b(t))^rdt\Big)^r\Big(\int_{\mathbb{R}\backslash I_{k_1}}|u|^sdt\Big)^{p/s}\\
    &=& \int_{I_{k_1}}\frac{b(t)}{p}|u_k|^pdt+\frac{1}{p}\|b\|_{L^r(\mathbb{R}\backslash I_{k_1})}\||u|^p\|_{L^{s/p}(\mathbb{R}\backslash I_{k_1},\mathbb{R}^{2N})}\\
    &<& \varepsilon\big(1+C/p\big)\quad \big(\textnormal{for $k$ big enough}\big).
\end{eqnarray*}
 It follows that
 \begin{equation*}
    \limsup_{\substack{k\rightarrow\infty}}\int_\mathbb{R}\frac{b(t)}{p}|u_k|^pdt\leq\varepsilon\big(1+C/p\big).
 \end{equation*}
 We then conclude by taking the limit $\varepsilon\rightarrow0$.
 \end{proof}

\subsection{The case of a general superquadratic condition}
In this subsection, we assume that $(A_0)$, $(W_1)-(W_4)$ and $(W_6)-(W_{10})$ are satisfied.\\
We define $\Phi_\lambda:X\rightarrow\mathbb{R}$ by
\begin{equation}\label{philamb}
    \Phi_\lambda(u)=\frac{1}{2}\|Qu\|^2-\lambda\Big[\frac{1}{2}\|Pu\|^2+\int_\mathbb{R}W(t,u)dt\Big],\quad \lambda\in[1,2].
\end{equation}
A standard argument shows that:
\begin{lemme}
The conditions $(B_1)$, $(B_2)$ and $(B_3)$ of Theorem \ref{vft} are satisfied, with
\begin{equation*}
    L(u):=\frac{1}{2}\|Qu\|^2,\quad J(u):=\frac{1}{2}\|Pu\|^2+\int_\mathbb{R}W(t,u)dt.
\end{equation*}
Moreover, $\Phi'_\lambda$ is given by
\begin{equation}\label{phiprime}
    \big<\Phi'_\lambda(u),v\big>=\big<Qu,v\big>- \lambda\Big[\big<Pu,v\big>+\int_\mathbb{R}v\cdot\nabla W(t,u)dt\Big].
\end{equation}
\end{lemme}
\begin{lemme}
For a.e. $\lambda\in[1,2]$, there exists $u_k(\lambda)\in X$ such that $\Phi_\lambda(u_k(\lambda))=c_k(\lambda)$ and $\Phi'(u_k(\lambda))=0$, for $k$ big enough.
\end{lemme}
\begin{proof}
The assumptions $(W_3)$ and $(W_7)$ imply that for every $\delta>0$, there is $C_\delta>0$ such that $W(t,u)\geq C_\delta|u|^\mu-\delta|u|^2$. It follows as in the proof of Theorem \ref{theo1} that for every $u\in Y_k$,  $\Phi_\lambda(u)\rightarrow-\infty$ as $\|u\|\rightarrow\infty$, uniformly in $\lambda\in[1,2]$. Therefore, we can choose $\rho_k$ sufficiently large such that $a_k(\lambda)\leq0$.
\par Let $u\in Z_k$. As in the proof of Theorem \ref{theo1}, we can show that for any $\lambda\in[1,2]$,
 \begin{equation*}
    \Phi_\lambda(u)\geq \frac{1}{2}\Big(\frac{1}{2}-C\beta_k^p\|u\|^p\Big),
 \end{equation*}
 where $C>0$ is constant and $\beta_k$ is defined by \eqref{betaka}. If $u$ is chosen such that $\|u\|=r_k:=\big(Cp\beta_k^p\big)^\frac{1}{2-p}$, then we have
 \begin{equation}\label{betild}
    \Phi_\lambda(u)\,\geq\, \widetilde{b_k}:=\frac{1}{2}\Big(\frac{1}{2}-\frac{1}{p}\Big)\big(Cp\beta_k^p\big)^\frac{2}{2-p}.
 \end{equation}
  Since by Lemma \ref{beta} $\beta_k\rightarrow0,$ we have $\widetilde{b_k}\rightarrow\infty$ and hence $b_k(\lambda)\rightarrow\infty$ uniformly in $\lambda$, as $k\rightarrow\infty$.\\
 By applying Theorem \ref{vft}, we then conclude, for $k$ large enough, that $c_k(\lambda)\geq b_k(\lambda)$ and for a.e. $\lambda\in[1,2]$, there exists a sequence $(v_k^n(\lambda))$ in $X$ such that
 \begin{equation*}
    \sup_{\substack{n}}\|v_k^n\|<\infty,\quad \Phi_\lambda(v_k^n)\rightarrow c_k(\lambda),\quad \Phi_\lambda(v_k^n)\rightarrow0,\,\,\textnormal{as }n\rightarrow\infty.
 \end{equation*}
We proceed as in the proof of Theorem \ref{theo1} to obtain the existence of $(u_k^n(\lambda))$ which satisfies the conclusion of the Lemma.
\end{proof}
As a consequence of the above lemma we have:
\begin{cor}\label{cor}
There exist $(\lambda_n)\subset[1,2]$ and $(z_k^n)_n\subset X\backslash\{0\}$ such that
\begin{equation*}
    \lambda_n\rightarrow1,\quad \Phi'_\lambda(z^n_k)=0,\quad \Phi_\lambda(z^n_k)=c_k(\lambda_n).
\end{equation*}
\end{cor}
We need the following lemma.
\begin{lemme}\label{lem}
Let $\lambda\in[1,2]$. If $z_\lambda\neq0$ and $\Phi'_\lambda(z_\lambda)=0$, then $\Phi_\lambda(z_\lambda+w)<\Phi_\lambda(z_\lambda)$ for every $w\in\mathcal{Z}_\lambda:=\{rz_\lambda+v\,; \, r\geq-1,v\in Y\}$.
\end{lemme}
\begin{proof}
Let $w=rz_\lambda+v\in Z_\lambda$. It is easy to verify that
\begin{multline}\label{eq7}
    \Phi_\lambda(z_\lambda+w)-\Phi_\lambda(z_\lambda)=-\frac{\lambda}{2}\|v\|^2+r(\frac{r}{2}+1)\|Qz_\lambda\|^2-\lambda r(\frac{r}{2}+1)\|Pz_\lambda\|^2\\
    -\lambda\Big[(1+r)\big<Pz_\lambda,v\big>+\int_\mathbb{R}W\big(t,(1+r)z_\lambda+v\big)dt-\int_\mathbb{R}W(t,z_\lambda)\Big].
\end{multline}
Now, $\Phi'_\lambda(z_\lambda)=0$ implies $\big<\Phi'_\lambda(z_\lambda),r(\frac{r}{2}+1)z_\lambda+(1+r)v\big>=0$, which gives
\begin{multline*}
    r(\frac{r}{2}+1)\|Qz_\lambda\|^2-\lambda\Big[r(\frac{r}{2}+1)\|Pz_\lambda\|^2+(1+r)\big<Pz_\lambda,v\big>\Big]=\\ \lambda\int_\mathbb{R}\big(r(\frac{r}{2}+1)z_\lambda+(1+r)v\big)\cdot \nabla W(t,z_\lambda).
\end{multline*}
Reporting this in \eqref{eq7} we obtain
\begin{multline}\label{eq}
     \Phi_\lambda(z_\lambda+w)-\Phi_\lambda(z_\lambda) = -\frac{\lambda}{2}\|v\|^2+\\
   \lambda\int_{\mathbb{R}}\big[\big(r(\frac{r}{2}+1)z_\lambda+
     (1+r)v\big)\cdot\nabla W(t,z_\lambda)+W(t,z_\lambda)-W(t,z_\lambda+w) \big]dt.
\end{multline}
We define $f:[-1,\infty[\rightarrow\mathbb{R}$ by
              \begin{equation*}
                f(s):=\big(s(\frac{s}{2}+1)z_\lambda+(1+s)v\big)\cdot\nabla W(t,z_\lambda)+W(t,z_\lambda)-W(t,z_\lambda+w).
              \end{equation*}
              Since $z_\lambda\neq0$, then in view of $(W_8)$ we have $f(-1)<0$. On the other hand, we deduce from $(W_7)$ and $(W_8)$ that $f(s)\rightarrow-\infty$ as $s\rightarrow\infty$. Therefore, $f$ attains its maximum at a point $s\in[-1,\infty[$ which satisfies
              \begin{equation}\label{fprim}
                f'(s)=\big((1+s)z_\lambda+v\big)\cdot\nabla W(t,z_\lambda)-z_\lambda\cdot\nabla W(t,(1+s)z_\lambda+v)=0.
              \end{equation}
              Setting $y_\lambda=z_\lambda+w=(1+s)z_\lambda+v$, one can easily verify that
              \begin{equation*}
                f(s)=-\big(\frac{s^2}{2}+s+1\big)z_\lambda\cdot\nabla W(t,z_\lambda)+(1+s)y_\lambda\cdot\nabla W(t,z_\lambda)+W(t,z_\lambda)-W(t,y_\lambda).
              \end{equation*}
              It is then clear that if $z_\lambda\cdot y_\lambda\leq0$, then $(W_6)$ and $(W_8)$ implies $f(s)<0$. Suppose that $z_\lambda\cdot y_\lambda>0$, then in view of \eqref{fprim}, $(W_{10})$ implies $|z_\lambda|=|y_\lambda|$ and  by $(W_9)$ we have $W(t,z_\lambda)=W(t,y_\lambda)$ and $y_\lambda\cdot\nabla W(t,z_\lambda)<z_\lambda\cdot\nabla W(t,z_\lambda)$, whenever $w\neq0$. This implies that $f(s)<-\frac{s^2}{2}z_\lambda\cdot\nabla W(t,z_\lambda)\leq0.$ Hence
              $f(r)<0$ for every $r\geq-1$.\\
              It then follows from \eqref{eq} that $\Phi_\lambda(z_\lambda+w)<\Phi_\lambda(z_\lambda)$.
\end{proof}
\begin{lemme}\label{bounded}
The sequence $(z_k^n)_n$ obtained in Corollary \ref{cor} above is bounded.
\end{lemme}
\begin{proof}
We assume by contradiction that $(z_k^n)$ is unbounded. Then, up to a subsequence, we may suppose that $\|z_k^n\|\rightarrow\infty$ as $n\rightarrow\infty$.
Let $w_k^n=z_k^n/\|z_k^n\|$, then since $(Qw_k^n)$ is bounded we have either
\begin{enumerate}
  \item [(i)] $(Qw_k^n)_n$ is vanishing, i.e.
  \begin{equation*}
    \lim_{\substack{n\rightarrow\infty}}\sup_{\substack{a\in\mathbb{R}}}\int_{a-1}^{a+1}|Qw_k^n|^2dt=0,
  \end{equation*}
  or
  \item [(ii)] $(Qw_k^n)_n$ is nonvanishing, i.e. there are numbers $r,\delta>0$ and a sequence $(a_n)\subset\mathbb{R}$ such that
  \begin{equation*}
    \liminf_{\substack{n\rightarrow\infty}}\int_{a_n-r}^{a_n+r}|Qw_k^n|^2dt\geq\delta.
  \end{equation*}
\end{enumerate}
Following an approach by Jeanjean \cite{LJJ}, we will find a contradiction by showing that neither $(i)$ nor $(ii)$ does not actually hold.
\par Assume that $(Qw_k^n)$ is vanishing. Then by Lemma \ref{lions} we have $Qw_k^n\rightarrow0$ as $n\rightarrow\infty$ in $L^p(\mathbb{R},\mathbb{R}^{2N})$. We then deduce from \eqref{eq2} that for every $R\geq0$, $\int_{\mathbb{R}}W(t,RQw_k^n)dt\rightarrow0$ as $n\rightarrow\infty.$ Since $\Phi_\lambda(z^n_k)\geq0$, we have $\|Qw_k^n\|\geq\|Pw_k^n\| $ and then $\|Qw_k^n\|^2\geq\frac{1}{2}$. By using Lemma $\ref{lem}$, we have
\begin{equation*}
    c_k(\lambda_n)=\Phi_{\lambda_n}(z_k^n)\geq \Phi_{\lambda_n}(RQw_k^n)\geq\frac{R^2}{4}-\lambda_n\int_{\mathbb{R}}W(t,RQw_k^n)dt.
\end{equation*}
Thus by setting $\widetilde{c_k}:=\sup_{u\in B_k}\Phi(u)$, we deduce that
\begin{equation*}
    \widetilde{c_k}\geq\frac{R^2}{4}-\lambda_n\int_{\mathbb{R}}W(t,RQw_k^n)dt\,\rightarrow R^2/4\quad\textnormal{as } n\rightarrow\infty.
\end{equation*}
 We obtain a contradiction by taking $R$ big enough.
 \par Assume now that $(Qw_k^n)$ is not vanishing. Then, up to a translation and a subsequence, we have
\begin{equation}\label{eq8}
    \liminf_{\substack{n\rightarrow\infty}}\int_{-r-\frac{1}{2}}^{r+\frac{1}{2}}|Qw_k^n|^2dt\geq\frac{\delta}{2}.
\end{equation}
Setting $w_k^n\rightharpoonup w_k$ as $n\rightarrow\infty,$ then  $(\ref{eq8})$ implies, since $Qw_k^n\rightarrow Qw_k$ in $L_{loc}^2(\mathbb{R},\mathbb{R}^{2N})$, that $Qw_k\neq0$. And this implies that $|z_k^n|\rightarrow\infty$ as $n\rightarrow\infty$. It then follows from $(W_7)$ and Fatou's lemma that
\begin{equation*}
    \liminf_{n\rightarrow\infty}\int_{\mathbb{R}}\frac{W(t,z_k^n)}{\|z_k^n\|^2}dt=\infty.
\end{equation*}
Hence,
\begin{equation*}
    0\leq\frac{\Phi_{\lambda_n}(z_k^n)}{\|z_k^n\|^2}=\frac{1}{2}\big(\|Qw_k^n\|^2-\lambda_n\|Pw_k^n\|^2\big)-
    \lambda_n\int_{\mathbb{R}}\frac{W(t,z_k^n)}{\|z_k^n\|^2}dt\rightarrow-\infty\quad \textnormal{as }n\rightarrow\infty.
\end{equation*}
A contradiction again.\\
Consequently, the sequence $(z_k^n)_n$ is bounded.
\end{proof}
We can now prove Theorem \ref{theo2}.
\begin{proof}[\textbf{Proof of Theorem \ref{theo2}}]
Consider the sequence $(z_k^n)$ above. One can easily verify that
\begin{equation*}
    \Phi(z_k^n)=\Phi_{\lambda_n}(z_k^n)+\frac{1}{2}(\lambda_n-1)\|Pz_k^n\|^2+(\lambda_n-1)\int_{\mathbb{R}}W(t,z_k^n)dt,
\end{equation*}
and
\begin{equation*}
    \big<\Phi'(z_k^n)-\Phi'_{\lambda_n}(z_k^n),v \big> =(\lambda_n-1)\Big[(Pz_k^n,v)+\int_{\mathbb{R}}v\cdot W(t,z_k^n)dt\Big].
\end{equation*}
Note that the sequence $(c_k(\lambda_n))_n$ is nondecreasing and bounded from above. Then, there is $c_k\geq c_k(1)\geq \widetilde{b_k}$ such that $c_k(\lambda_n)\rightarrow c_k$ as $n\rightarrow\infty$ \big(where $\widetilde{b_k}$ is defined in \eqref{betild}\big). It follows from the above relations that $(z_k^n)$ is a $(PS)_{c_k}$ sequence for $\Phi$. By repeating the argument in the proof of Theorem \ref{theo1}, we obtain the existence of $z_k\in X$ such that $\Phi'(z_k)=0$ and $\Phi(z_k)\geq \widetilde{b_k}$. Since $\widetilde{b_k}\rightarrow\infty$ as $k\rightarrow\infty$, the proof of Theorem \ref{theo2} is completed.
\end{proof}
\section{Existence of a ground state homoclinic solution}\label{section3}
\subsection{Generalized Nehari manifold}
Let $X$ be a Hilbert space with norm $\|\cdot\|$, and an orthogonal decomposition $X=X^+\oplus X^-$. We denote by $S^+$ the unit sphere in $X^+$; that is,
\begin{equation*}
    S^+:=\big\{u\in X^+\,\big|\,\|u\|=1\big\}.
\end{equation*}
For $u=u^+ + u^-\in X$, where $u^\pm\in X^\pm$, we define
\begin{equation}\label{}
    X(u):=\mathbb{R}u\oplus X^-\equiv\mathbb{R}u^+\oplus X^- \,\,\, \textnormal{and}\,\,\, \widehat{X}(u):=\mathbb{R}^+u\oplus X^-\equiv\mathbb{R}^+u^+\oplus X^-.
\end{equation}
Let $\Phi$ be a $\mathcal{C}^1-$functional defined on $X$ by
\begin{equation*}
    \Phi(u):=\frac{1}{2}\|u^+\|^2-\frac{1}{2}\|u^-\|^2-P(u).
\end{equation*}

We consider the following situation:

\begin{enumerate}
  \item [$(H_1)$]  $P(0)=0$, $\frac{1}{2}\big<P'(u),u\big>>P(u)>0$ for all $u\neq0$ and $P$ is weakly lower semicontinuous.
  \item [$(H_2)$] For each $w\in X\backslash X^-$, there exists a unique nontrivial critical point of $\widehat{m}(w)$ of $\Phi|_{\widehat{X}(w)}$, which is the unique global maximum of $\Phi|_{\widehat{X}(w)}$.
  \item [$(H_3)$] There exists $\delta>0$ such that $\|\widehat{m}(w)^+\|\geq\delta$ for all $w\in X\backslash X^-$, and for each compact subset $\mathcal{K}\subset X\backslash X^-$, there exists a constant $C_\mathcal{K}$ such that $\|\widehat{m}(w)\|\leq C_\mathcal{K}$.
\end{enumerate}
\par The following set was introduced by Pankov \cite{Pan}:
\begin{equation*}
    \mathcal{M}:=\big\{u\in X\backslash X^-: \big<\Phi'(u),u\big>=0 \,\, \textnormal{and} \,\, \big<\Phi'(u),v\big>=0\,\, \forall v\in X^-\big\}.
\end{equation*}
It is called the generalized Nehari manifold.
\begin{rem}
 \textnormal{By $(H_1)$, $\mathcal{M}$ contains all nontrivial critical points of $\Phi$ and by $(H_2)$, $\widehat{X}(w)\cap\mathcal{M}=\{\widehat{m}(w)\}$ whenever $w\in X\backslash X^-$.}
 \end{rem}
\par We also consider the mappings:
\begin{equation*}
    \widehat{m}:X\backslash X^-\rightarrow \mathcal{M}, \, w\mapsto \widehat{m}(w)\,\,\,\, \textnormal{and }\,\,\, m:=\widehat{m}|_{S^+}.
\end{equation*}
\begin{equation*}
    \widehat{\Psi}:X^+\backslash\{0\}\rightarrow \mathbb{R},\, \widehat{\Psi}(w):=\Phi(\widehat{m}(w)) \,\, \textnormal{ and }\,\, \Psi:=\widehat{\Psi}|_{S^+}.
\end{equation*}
The following result is due to A. Szulkin and T. Weth \big(\cite{S-W}, Corollary $33$\big).
\begin{theo}\label{reduction}
If $(H_1),$ $(H_2)$ and $(H_3)$ are satisfied, then
\begin{itemize}
  \item [(a)] $\Psi\in\mathcal{C}^1(S^+,\mathbb{R})$ and
  \begin{equation*}
    \big<\Psi'(w),z\big>=\|m(w)^+\|\big<\Phi'(m(w)),z\big> \textnormal{ for all }z\in T_w(S),
  \end{equation*}
  where $T_w(S)$ is the tangent space of $S$ at $w$.
  \item [(b)] If $(w_n)$ is a Palais-Smale sequence for $\Psi,$ then $(m(w_n))$ is a Palais-Smale sequence for $\Phi$. If $(u_n)\subset\mathcal{M}$ is a bounded Palais-Smale sequence for $\Phi$, then $(m^{-1}(w_n))$ is a Palais-Smale sequence for $\Psi$.
  \item [(c)] $w$ is a critical point of $\Psi$ if and only if $m(w)$ is a nontrivial critical point of $\Phi$. Moreover, the corresponding critical values coincide and $inf_{S^+}\Psi=inf_\mathcal{M}\Phi$.
\end{itemize}
\end{theo}

\subsection{Existence of a ground state}
Throughout this subsection, we assume that $(A_0)$, $(W_1)-(W_3)$ and $(W_6)-(W_{10})$ are satisfied. \\
Here $P$ is given by
\begin{equation*}
    P(u):=\int_\mathbb{R}W(t,u)dt.
\end{equation*}
\begin{lemme}
Condition $(H_1)$ is satisfied.
\end{lemme}
\begin{proof}
Clearly we have $P(0)=0$ and $\frac{1}{2}\big<P'(u),u\big>\geq P(u)>0$ for any $u\neq0$. Let $(u_n)\subset X$ and $C\in\mathbb{R}$ such that $u_n\rightharpoonup u$ and $N(u_n)\leq C$. Since the embedding of $X$ in $L^2_{loc}(\mathbb{R},\mathbb{R}^{2N})$ is compact, we have $u_n\rightarrow u$ in $L^2_{loc}(\mathbb{R},\mathbb{R}^{2N})$ and up to a subsequence $u_n\rightarrow u$ a.e.. It then follows from Fatou's lemma that $P(u)\leq C$. Hence $P$ is weakly lower semicontinuous.
\end{proof}
\begin{lemme}
Condition $(H_2)$ is satisfied.
\end{lemme}

\begin{proof}
 Let $w\in X\backslash X^-.$ Then there exists $R$ large enough such that $\Phi\leq0$ on $\widehat{X}(w)\backslash B_R$, where $B_R:=\{u\in X\, |\, \|u\|\leq R\}$. In fact, if this is not true then there exists a sequence $(u_n)\subset\widehat{X}(w)$ such that $\|u_n\|\rightarrow\infty$ and $\Phi(u_n)>0$. Up to a subsequence we have $v_n=u_n/\|u_n\|\rightharpoonup v$ in $X$.
              By \eqref{phi} we have
              \begin{equation*}
                0<\frac{\Phi(u_n)}{\|u_n\|^2}=\frac{1}{2}\|v_n^+\|^2-\frac{1}{2}\|v_n^-\|^2-\int_\Omega\frac{F(x,\|u_n\|v_n)}{\big|v_n\|u_n\|\big|^2}|v_n|^2.
              \end{equation*}
If $v\neq0$, we deduce by using Fatou's Lemma and $(W_7)$ that $0\leq-\infty$; a contradiction. Consequently $v=0$. Since $\widehat{X}(w)=\widehat{X}(w^+/\|w^+\|)$, we may assume that $w\in S^+$. Now, since $P(u_n)\geq0$ and $1=\|v_n^+\|^2+\|v_n^-\|^2$, then necessarily $v_n^+=s_nw\nrightarrow0$. Hence there is $r>0$ such that $\|v_n^+\|=\|s_nw\|>r$ $\forall n$. So $\|v_n^+\|=s_n$ is bounded and bounded away from $0$. But then, up to a subsequence, $v_n^+\rightarrow sw,$ $s>0$, which contradicts the fact that $v_n\rightharpoonup0$.\\
By $(W_3)$, $\Phi(sw)=\frac{1}{2}s^2+\circ(s^2)$ as $s\rightarrow0$. Hence $0<\sup_{\widehat{X}(w)}\Phi<\infty$. Since $\Phi$ is weakly upper semicontinuous on $\widehat{X}(w)$ and $\Phi\leq0$ on $\widehat{X}(w)\cap X^-$, the supremum  is attained at some point $u_0$ such that $u_0^+\neq0$. So $u_0$ is a nontrivial critical point of $\Phi|_{\widehat{X}(w)}$ and hence $u_0\in \mathcal{M}$.
\par We will now show that if $u\in \mathcal{M}$, then $u$ is the unique global maximum of $\Phi|_{\hat{X}(u)}$.
Let $u\in\mathcal{M}$ and $w=u+w\in\hat{X}(u)$ with $w\neq0$. By the definition of $\hat{X}(u)$, we have $u+w=(1+s)u+v$, with $s\geq-1$ and $v\in X^-$. Using the argument in the proof of Lemma \ref{lem}, we see that $\Phi(u+w)<\Phi(u)$.
\end{proof}
\begin{lemme}
Condition $(H_3)$ is satisfied.
\end{lemme}
\begin{proof}
$(W_3)$ and $(H_1)$ imply that
\begin{equation*}
    \forall \varepsilon>0,\,\exists\alpha>0,\, |u|<\alpha\, \Longrightarrow\, P(u)<\frac{1}{2}\big<P'(u),u\big>\leq\frac{\varepsilon}{2}\|u\|^2.
\end{equation*}
Hence, there exist $\rho,\eta>0$ such that $\Phi(w)\geq\eta$ for every $w\in\{u\in X^+\,|\,\|u\|=\rho\}$. By $(H_2)$, we have $\Phi(\hat{m}(w))\geq\eta$ for every $w\in X\backslash X^-$. Since $P\geq0$, we deduce from \eqref{phi} that $\|\hat{m}(w)^+\|\geq\sqrt{2\eta}$ $\forall w\in X\backslash X^-$.
\par Now let $\mathcal{K}$ be a compact subset of $X\backslash X^-$. We claim that there exists a constant $C_\mathcal{K}$ such that $\|\widehat{m}(w)\|\leq C_\mathcal{K}$, $\forall w\in\mathcal{K}$. In fact, if the claim is not true, then we can find a subsequence $(w_n)\subset\mathcal{K}$ such that $\|\widehat{m}(w_n)\|\rightarrow\infty$ as $n\rightarrow\infty$. Since $\widehat{m}(w)=\widehat{m}(w^+/\|w^+\|)$ $\forall w\in X\backslash X^-$, we may assume that $\mathcal{K}\subset S^+$. By using the fact that $\widehat{m}(w_n)\in\mathcal{M}$, one can verify easily that $\Phi(\widehat{m}(w_n))>0$. Define $y_n=\widehat{m}(w_n)/\|\widehat{m}(w_n)\|$. Then we have
\begin{equation*}
    0\leq \frac{\Phi(\widehat{m}(w_n))}{\|\widehat{m}(w_n)\|^2}=\frac{1}{2}\Big(\frac{\|\widehat{m}(w_n)^+\|^2}{\|\widehat{m}(w_n)\|^2}-\frac{\|\widehat{m}(w_n)^-\|^2}{\|\widehat{m}(w_n)\|^2}\Big)-
    \int_\mathbb{R}\frac{W(t,y_n\|\widehat{m}(w_n)\|)}{|y_n\|\widehat{m}(w_n)\||^2}|y_n|^2dt.
\end{equation*}
Since $y_n\in \hat{X}(u)$, then $y_n=s_nw_n+v_n$, with $s_n\geq0$ and $v_n\in X^-$. It follows that
\begin{equation}\label{eq9}
    0\leq \frac{\Phi(\widehat{m}(w_n))}{\|\widehat{m}(w_n)\|^2}=\frac{1}{2}\big(\lambda_n^2-\|v_n\|^2\big)-
    \int_\mathbb{R}\frac{W(t,y_n\|\widehat{m}(w_n)\|)}{|y_n\|\widehat{m}(w_n)\||^2}|y_n|^2dt.
\end{equation}
Since $W\geq0$, we deduce that $s_n^2\geq\|v_n\|^2$ and then $\frac{1}{\sqrt{2}}\leq s_n\leq1$. Up to a subsequence, $s_n\rightarrow s>0$ and $w_n\rightarrow w\in S^+$. Hence $y_n\rightharpoonup y\neq0$. If we take the limit $n\rightarrow \infty$ in \eqref{eq9}, we obtain by using $(W_7)$ and Fatou's lemma the contradiction $0\leq-\infty$.
\end{proof}
\begin{lemme}\label{cont}
There exists $\alpha>0$ such that
\begin{equation*}
    c=\inf_{\substack{\mathcal{M}}}\Phi\geq\inf_{\substack{S_\alpha}}\Phi>0,
\end{equation*}
where $S_\alpha:=\{u\in X^+\,;\, \|u\|=\alpha\}$.
\end{lemme}
\begin{proof}
We remark that we can choose $\varepsilon$ in \eqref{eq2} in such a way that
\begin{equation*}
    \Phi(u)\geq \frac{1}{4}\|u\|^2-C\|u\|^p,\quad \forall u\in X^+.
\end{equation*}
It suffices to take $\alpha$ sufficiently small.
\end{proof}
\begin{proof}[\textbf{Proof of Theorem \ref{theo3}}]
By the preceding lemmas we know that $(H_1)$, $(H_2)$ and $(H_3)$ are satisfied. By Theorem \ref{reduction}-$(a)$, $\Psi\in\mathcal{C}^1(S^+,\mathbb{R})$. The Ekeland variational principle \cite{W} then gives the existence of a sequence $(w_n)\subset S^+$ such that $\Psi(w_n)\rightarrow\inf_{S^+}\Psi$. By Theorem \ref{reduction}-$(b)$, $\big(u_n:=\hat{m}(w_n)\big)$ is a Palais-Smale sequence for $\Phi$ on $\mathcal{M}$. By using the argument in the proof of Lemma \ref{bounded}, we show that $(u_n)$ is bounded. Up to a subsequence, $u_n\rightharpoonup u$ in $X$. We claim that $u_n\nrightarrow0$ in $L^p(\mathbb{R},\mathbb{R}^{2N})$. In fact, if this is not true, then we deduce from \eqref{eq2} that $\int_\mathbb{R}u_n^+\cdot W(t,u_n)dt\rightarrow0$ as $n\rightarrow\infty$. It follows that
 \begin{equation*}
    \|u_n^+\|^2=\big<\Phi'(u_n),u_n^+\big>-\int_\mathbb{R}u_n^+\cdot W(t,u)dt\rightarrow0 \quad \textnormal{as }n\rightarrow\infty.
 \end{equation*}
 But since $\Phi(u_n)\leq \frac{1}{2}\|u_n^+\|^2$, we deduce that $\liminf_{n\rightarrow\infty}\Phi(u_n)=0$, which contradicts Lemma \ref{cont}. Hence $u_n\nrightarrow0$ in $L^p(\mathbb{R},\mathbb{R}^{2N})$.\\
 By Lemma \ref{lions}, there exist $\delta>0$ and $(a_n)\subset\mathbb{R}_+$ such that
 \begin{equation*}
    \int_{-a_n}^{a_n}|u_n|^2dt\geq\delta.
 \end{equation*}
 Up to translation and a subsequence, we deduce that $u\neq0$. Now since $\Phi'$ is weakly sequentially continuous, we obtain $\Phi'(u)=0$, that is, $u$ is a non trivial weak solution of \eqref{hamilto}. On the other hand, we may assume that $u_n\rightarrow u$ a.e., which together with $(W_8)$ and Fatou's lemma imply, since
 \begin{equation*}
    \Phi(u_n)-\frac{1}{2}\big<\Phi'(u_n),u_n\big>\, =\, \int_\mathbb{R}\Big(\frac{1}{2}u_n\cdot\nabla W(t,u_n)-W(t,u_n)\Big)dt,
 \end{equation*}
that
\begin{equation*}
    \inf_{\substack{\mathcal{M}}}\Phi\geq\int_\mathbb{R}\big(\frac{1}{2}u\cdot\nabla W(t,u)-W(t,u)\big)dt=\Phi(u)-\frac{1}{2}\big<\Phi'(u),u\big>.
\end{equation*}
This implies $\Phi(u)\leq \inf_{\substack{\mathcal{M}}}\Phi$. Since $u\in \mathcal{M}$, the reverse inequality also holds and therefore
\begin{equation*}
    \Phi(u)= \inf_{\substack{\mathcal{M}}}\Phi.
\end{equation*}
\end{proof}

\end{document}